\documentclass[12pt,reqno]{amsart}
\usepackage{amscd,amssymb,amsmath,amsthm}
\usepackage[arrow,matrix]{xy}
\usepackage{graphicx}
\usepackage{cite}
\usepackage{geometry}
\newtheorem{thm}{Theorem}
\newtheorem{defn}{Definition}
\newtheorem{lem}{Lemma}
\newtheorem{pro}{Proposition}

\numberwithin{equation}{section}

\begin{document}
\title[Subgroups and normal subgroups]
{On the subgroups and normal subgroups of the group representation
of the Cayley tree.}

\author{F. H. Haydarov}
 \address{F.H.Haydarov\\ National University of Uzbekistan,
Tashkent, Uzbekistan.}
 \email {haydarov\_ imc@mail.ru.}

\begin{abstract}

In this paper, we give a characterization of the normal subgroups
of index $2^{s}(2n+1),\ s\in\{1,2\},\ n\in \mathbb{N}$ and of the
subgroups of index three of the group representation of the Cayley
tree.

\end{abstract}

\maketitle



\section{Introduction}

   In the theory of groups there are very important unsolved problems,
   most of which arise in solving of problems of
  natural sciences as physics, biology etc.
  In particular, if configuration of a physical system is
  located on a lattice (in our case on the graph of a group)
  then the configuration can be considered as a function defined on
  the lattice. Usually, more important configuration
  (functions) are the periodic ones. It is well-known that if
  the lattice has a group representation then periodicity
  of a function can be defined by a given subgroup of the
  representation. More precisely, if a subgroup, say $H$, is given,
  then one can define a $H$- periodic function as a function,
  which has a constant value (depending only on the coset) on
  each (right or left) coset of $H$. So the periodicity is
  related to a special partition of the group
  (that presents the lattice on which our physical system is
  located). There are many works devoted to several kinds of
  partitions of groups (lattices)
  (see e.g. \cite{1},\cite{3},\cite{5},\cite{7}).\\

   One of the central problems in the theory of Gibbs measures is to
study periodic Gibbs measures corresponding to a given Hamiltonian
of model. For any subgroup $H$ of the group $G_{k}$ we define
$H$-periodic Gibbs measures. To find new periodic and weakly
periodic Gibbs measures one usually needs to find new subgroups of
the group representation of the Cayley tree. In Chapter 1 of
\cite{5} it is given a one to one correspondence between the set
of vertices $V$ of the Cayley tree $\Gamma^{k}$ and the group
$G_{k}$ and it is given a full description of subgroups of index
two and there are also constructed several normal subgroups of the
group $G_{k}$. But there wasn't a full description of normal
subgroups of finite index (without index two) and to the best of
our knowledge until now there wasn't any description of a not
normal subgroup of finite index of the group representation of the
Cayley tree. In \cite{4} it is given a full description of normal
subgroups of indices 4 and 6 for the group representation of the
Cayley tree. In this paper we continue this investigation and
construct all normal subgroups of index $2^{s}(2n+1),\
s\in\{1,2\},\ n\in \mathbb{N}$ and all subgroups of
index three for the group representation of the Cayley tree. \\

{\it Cayley tree and its group representation.}   \,\  A Cayley
tree (Bethe lattice) $\Gamma^k$ of order $k\geq 1$ is an infinite
homogeneous tree, i.e., a graph without cycles, such that exactly
$k+1$ edges originate from each vertex. Let $\Gamma^k=(V,L)$ where
$V$ is the set of vertices and $L$ that of edges (arcs).
  Let $G_{k}$ be a free product of $k+1$ cyclic groups of the second order with
generators $a_{1},a_{2},...a_{k+1},$ respectively.
 It is known that there exists a one to one correspondence between
the set of vertices $V$ of the Cayley tree $\Gamma^{k}$ and the
group $G_{k}$. To give this correspondence we fix an arbitrary
element $x_{0}\in V$ and let it correspond to the unit element $e$
of the group $G_{k}.$ Using $a_{1},...,a_{k+1}$ we label the
nearest-neighbors of element $e$, moving in positive direction.
Now we'll label the nearest-neighbors of each $a_{i}, i=1,...,k+1$
by $a_{i}a_{j}, j=1,...,k+1$. Since all $a_{i}$ have the common
neighbor $e$ we have $a_{i}a_{i}=a_{i}^{2}=e.$ Other neighbors are
labeled starting from $a_{i}a_{i}$ in positive direction. We label
the set of all the nearest-neighbors of each $a_{i}a_{j}$ by words
$a_{i}a_{j}a_{q}, q=1,...,k+1$ starting from
$a_{i}a_{j}a_{j}=a_{i}$ by the positive direction. Iterating this
argument one gets a one-to-one correspondence between the set of
vertices $V$ of the Cayley tree
$\Gamma^{k}$ and the group $G_{k}.$\\

 Any(minimal represented) element $x\in G_{k}$ has the following form:
\,\ $x=a_{i_{1}}a_{i_{2}}...a_{i_{n}},$ where $1\leq i_{m}\leq
k+1, m=1,...,n.$  The number $n$ is called the length of the word
$x$ and is denoted by $l(x).$  The number of letters $a_{i},
i=1,...,k+1,$ that enter the non-contractible representation of
the word $x$ is denoted by $w_{x}(a_{i}).$

 The following result is well-known in group theory. If $\varphi$
 is a homomorphism of a group $G$ with the kernel $H$ then $H$ is a
normal subgroup of the group $G$ and $\varphi(G)\simeq G/H,$
(where $G/H$ is the quotient group) so the index $|G:H|$ coincides
with the order $|\varphi(G)|$ of the group $\varphi(G).$

Usually we define natural homomorphism $g$ from $G$ onto the
quotient group $G/H$ by the formula $g(a)=aH$ for all $a\in G.$
Then $Ker\varphi=H.$

\begin{defn}\label{d0}\ Let $M_{1}, M_{2},..., M_{m}$ be some sets and
$M_{i}\neq M_{j},$ for $i\neq j.$ We call the intersection
$\cap_{i=1}^{m}M_{i}$ contractible if there exists $i_{0} (1\leq
i_{0}\leq m)$ such that
 $$\cap_{i=1}^{m}M_{i}=\left(\cap_{i=1}^{i_{0}-1}M_{i}\right)\cap\left(\cap_{i=i_{0}+1}^{m}M_{i}\right).$$
\end{defn}

  Let $N_{k}=\{1,...,k+1\}.$ The following Proposition describes
  several normal subgroups of $G_{k}.$

Put \begin{equation} H_{A}=\left\{x\in G_{k}\,\ | \,\ \sum_{i\in
A}\omega_{x}(a_{i}) \ {\it is \ even} \right\},\ \ A\subset N_{k}.
\end{equation}

\begin{pro}\label{p2} \cite{5} For any $\emptyset \neq A \subseteq
  N_{k},$ the set $H_{A}\subset G_{k}$ satisfies the
  following properties:
 (a) $H_{A}$ is a normal subgroup and $|G_{k}:H_{A}|=2;$\\
 (b) $H_{A}\neq H_{B}$,\ for all $A\neq B \subseteq N_{k};$\\
 (c) Let $A_{1}, A_{2},...,A_{m}
 \subseteq N_{k}.$ If $\cap_{i=1}^{m}H_{A_{i}}$ is non-contractible, then it is
 a  normal subgroup of index $2^{m}.$\end{pro}

\begin{thm}\label{th0} \cite{5}

\item 1. The group $G_{k}$ does not have normal subgroups of odd
index $(\neq 1)$.

\item 2. The group $G_{k}$ has normal subgroups of arbitrary even index.
\end{thm}

\section{Subgroups and normal subgroups of finite index.}

\subsection{Normal subgroups of index $\mathbf{2^{s}(2n+1)},\ \mathbf{s\in\{1,2\}},\ \mathbf{n\in \mathbb{N}}$.}

\begin{defn}\label{d9} A commutative group $G$ of order $2^{n}$ is called $K_{2^{n}}-group$ if $G$ is generated by free product of $n$ elements
$c_{i}, i=\overline{1,n}$ satisfying the relations $o(c_{i})=2,\
i=\overline{1,n},$
\end{defn}

 \begin{pro}\label{p1.} Let $\varphi$ be a homomorphism of the group $G_{k}$ onto a finite commutative group
 $G.$ Then $\varphi(G_{k})$ is isomorphic to $K_{2^{i}}$ for some $i\in \mathbb{N}.$
\end{pro}

\begin{proof} Let $(G,\ast)$ be a commutative group of order
$n$ and $\varphi:G_{k}\rightarrow G$ be an epimorphism. We will
first show that $n\in \{2^{i}|\ i\in \mathbb{N}\}.$ Suppose $n$
does not belong to $\{2^{i}|\ i\in \mathbb{N}\}.$ Then there exist
$m>1,\ m,s \in \mathbb{N}\cup \{0\}$ such that $n=2^{s}m,$ where
$m$ is odd. Since $\varphi: G_{k}\rightarrow G$ is an epimorphism
there exist distinct elements $\varphi(a_{j_{p}}),\
p=\overline{1,s}.$ Clearly, $<\varphi(a_{j_{1}}),
 \varphi(a_{j_{2}}) ... ,\varphi(a_{j_{s}})>$ is a subgroup of $G$ and
$|<\varphi(a_{j_{1}}), \varphi(a_{j_{2}}) ...
,\varphi(a_{j_{s}})>|=2^{s}.$ Since $m>1$ there exist at least one
element $a_{j_{0}}\in G_{k}$ such that $\varphi(a_{j_{0}})\neq
\varphi(a_{j_{i}}), \ i\in \{1,2,...,s\}.$ Hence
$<\varphi(a_{j_{0}}),
 \varphi(a_{j_{1}}) ... ,\varphi(a_{j_{s}})>$ is a subgroup of $G.$
Then by Lagrange's theorem $|G|: 2^{s+1}$ which contradicts our
assumption that $m$ is odd. Thus $n\in \{2^{i}|\ i\in
\mathbb{N}\}.$ If $n=2^{q}$ then $G$ is equal to
$<\varphi(a_{j_{1}}),
 \varphi(a_{j_{2}}) ... ,\varphi(a_{j_{q}})>$  We define the following mapping
$\gamma(a_{j_{i}})=c_{i}\in K_{q}$ for all $i\in\{1,2,...,q\}.$ It
is easy to check that this mapping is a one to one correspondence
from $G$ to $K_{q}.$ This completes the proof.
\end{proof}

 Let $A_{1},A_{2},...,A_{m}\subset N_{k},\ m\in \mathbb{N}$ and $\cap_{i=1}^{m}H_{A_{i}}$
is non-contractible. Then we denote by $\Re$ the following
set\\
$$\Re=\{\cap_{i=1}^{m}H_{A_{i}}| \ A_{1},A_{2},...,A_{m}\subset
N_{k},\ m\in \mathbb{N}\}.$$

\begin{thm}\label{th1}  Let $\varphi$ be a homomorphism from $G_{k}$ to a finite commutative group.
Then there exists an element $H$ of $\Re$ such that
$Ker\varphi\simeq H$ and conversely.
\end{thm}

 \begin{proof} We give a one to one correspondence between $\left\{Ker\varphi|\
 \varphi(G_{k})\simeq G\in\{K_{2^{i}}|\ i\in \mathbb{N}\}\right\}$ and $\Re.$
Let $\varphi$ be a homomorphism from $G_{k}$ to a finite
commutative group of order $p$. Then by Proposition \ref{p1.} the
number $p$ belongs to the set $\{2^{i}|\ i\in \mathbb{N}\}$ and if
$p=2^{n}$ then $\varphi:G_{k}\rightarrow K_{2^{n}}$ is an
isomorphism. For any nonempty sets $A_{1}, A_{2},...,A_{n}\subset
N_{k}$ and $a_{i}\in G_{k},\ i\in N_{k}$ we define the following
homomorphism\\
$$\phi_{A_{1}A_{2}...A_{n}}(a_{i})=\left\{\begin{array}{ll}
c_{j}, \ \ \mbox{if} \ \ i\in A_{j}\setminus(A_{2}\cup A_{3}\cup...\cup A_{n}),\ j=\overline{1,n}\\[2.5mm]
.\ .\ .\ .\ .\ .\ .\ .\ \\[2.5mm]
c_{j_{1}}c_{j_{2}}, \ \ \mbox{if} \ \ i\in (A_{j_{1}}\cap
A_{j_{2}})\setminus(A_{1}\cup...A_{j_{1}-1}\cup
A_{j_{1}+1}\cup...\\[2.5mm]
...\cup A_{j_{2}-1}\cup A_{j_{2}+1}\cup...\cup A_{n}), \ 1\leq j_{1}<j_{2}\leq n\\[2.5mm]
.\ .\ .\ .\ .\ .\ .\ .\ \\[2.5mm]
c_{1}...c_{j-1}c_{j+1}...c_{n}, \ \ \mbox{if} \ \ i\in (A_{1}\cap...\cap A_{j-1}\cap A_{j+1}...\cap A_{n})\setminus A_{j},\ j=\overline{1,n}\\[2.5mm]
c_{1}c_{2}...c_{n}, \ \ \mbox{if} \ \ i\in (A_{1}\cap A_{2}\cap...\cap A_{n})\\[2.5mm]
e, \ \ \mbox{if} \ \ i\in N_{k}\setminus(A_{1}\cup
A_{2}\cup...\cup A_{n}). \\
\end{array}\right.$$\\
If $i\in \emptyset$ then we'll accept that there is no index $i\in
N_{k}$ satisfying the required condition. For $x\in G_{k}$ if
$\varphi_{A_{1}A_{2}... A_{n}}(x)=e$ the number of
$a_{i},\ i=\overline{1,n}$ of $a_{i}$ appearing in the word $x$ must be even.
Therefore\\
$$Ker\varphi_{A_{1}A_{2}... A_{n}}=H_{A_{1}}\cap H_{A_{2}}\cap...\cap H_{A_{n}}.$$
Thus the following equality holds (up to isomorphism)\\
$$\left\{Ker\varphi|\
 \varphi(G_{k})\simeq G\in\{K_{2^{i}}|\ i\in \mathbb{N}\}\right\}= \Re.$$
 \end{proof}

The group $G$ has finitely generators of the order two. Assume
that $r$ is the minimal number of such generators of the group $G$
and without loss of generality we can take these generators to be
$b_{1},b_{2},...b_{r}.$ Let $e_{1}$ be the identity element of the
group $G.$ We define a homomorphism from $G_{k}$ onto $G.$ Let
$\Xi_{n}=\{A_{1}, \,\ A_{2},...,A_{n}\}$ be a partition of
$N_{k}\backslash A_{0}, \,\ 0\leq|A_{0}|\leq k+1-n.$ Then we
consider the homomorphism
$u_{n}:\{a_{1},a_{2},...,a_{k+1}\}\rightarrow \{e_{1},
b_{1}...,b_{m}\}$ given by\\
 \begin{equation} u_{n}(x)=\left\{\begin{array}{ll}
e_{1}, \ \ \mbox{if} \ \ x=a_{i}, i\in A_{0}\\[2.5mm]
b_{j}, \ \mbox{if}\ \ x=a_{i}, i\in A_{j}, j=\overline{1,n}.
\\ \end{array}\right.\end{equation}\\
For $b\in G$ we denote by $R_{b}[ b_{1}, b_{2}, ..., b_{m}]$ a
representation of the word $b$ by generators $b_{1}, b_{2}, ...,
b_{r}, \ r\leq m.$  Define the homomorphism $\gamma_{n}:
G\rightarrow G$ by the formula\\
\begin{equation}\gamma_{n}(x)=\left\{\begin{array}{ll}
e_{1}, \ \ \mbox{if} \ \ x=e_{1}\\[2.5mm]
b_{i}, \ \mbox{if}\ \ x=b_{i}, i=\overline{1,r}\\[2.5mm]
R_{b_{i}}[b_{1},...,b_{r}], \ \mbox{if}\ \ x=b_{i}, \ i\notin\{1,...,r\}.\\
\end{array}\right.\end{equation}
Put
\begin{equation} H^{(p)}_{\Xi_{n}}(G)=\{x\in G_{k}|\,\
l(\gamma_{n}(u_{n}(x))):2p\},\ 2\leq n \leq k-1. \end{equation}

Let $\gamma_{n}(u_{n}(x)))=\tilde{x}.$ We introduce the following
equivalence relation on the set $G_{k}:$ $x\sim y$ if
$\tilde{x}=\tilde{y}.$ It's easy to check this relation is
reflexive, symmetric and transitive.

 Let $(G,\ast)$ be a group of order $n$, generated by two elements
  of order 2, and let $\Im_{n}$ be the set of all such groups.

\begin{pro}\label{pr1} \  For the group
$G_{k}$ the following equality holds\\
$$\{Ker \varphi|\ \varphi:G_{k}\rightarrow G\in \Im_{2n}\ is\
 an\ epimorphism\}=$$
$$=\{H^{(n)}_{B_{0}B_{1}B_{2}}(G)|\ B_{1},
B_{2} \ is\ a\ partition\ of\ the\ set\ N_{k}\setminus B_{0},\
0\leq |B_{0}|\leq k-1\}.$$
\end{pro}

\begin{proof} For an arbitrary group $G\in \Im_{2n}$ we give a one to one
correspondence between the two given sets.
 Let $e_{1}$ be the identity element of $G$ and $B_{1}, B_{2}$ be a
 partition of $N_{k}\setminus
B_{0},$ where $B_{0}\subset N_{k},\ 0\leq |B_{0}|\leq k-1.$ Then
we define the homomorphism
$\varphi_{B_{0}B_{1}B_{2}}:G_{k}\rightarrow G$ by the formula\\
\begin{equation}
\varphi_{B_{0}B_{1}B_{2}}(a_{i})=\left\{\begin{array}{ll}
b_{1}, \ \ \mbox{if} \ \ i\in B_{1}\\[2.5mm]
b_{2}, \ \mbox{if}\ \ i\in B_{2}.
\\ \end{array}\right.\end{equation}\\
There is only one such a homomorphism (corresponding to $B_{0},
B_{1}, B_{2}$). It's clearly $x\in Ker\varphi_{B_{0}B_{1}B_{2}}$
if and only if $\tilde{x}$ is equal to $e_{1}.$ Hence it is
sufficient to show that if $y\in H^{(n)}_{B_{0}B_{1}B_{2}}(G)$
then $\tilde{y}=e_{1}.$ Suppose that there exist $y\in G_{k}$ such
that $l(\tilde{y})\geq 2n.$ Let
$\tilde{y}=b_{i_{1}}b_{i_{2}}...b_{i_{q}}, \ q\geq 2n$ and
$S=\{b_{i_{1}}, b_{i_{1}}b_{i_{2}},...,
b_{i_{1}}b_{i_{2}}...b_{i_{q}}\}.$ Since $S\subseteq G$ there
exist $x_{1}, x_{2} \in S$ such that $x_{1}=x_{2}$ which
contradicts the fact that $\tilde{y}$ is a non-contractible. Thus
we have proved that $l(\tilde{y})<2n.$ Since $y\in
H^{(n)}_{B_{0}B_{1}B_{2}}(G)$ the number $l(\tilde{y})$ must be
divisible by $2n.$ Thus for any $y\in
H^{(n)}_{B_{0}B_{1}B_{2}}(G)$ we have $\tilde{y}=e_{1}.$ Hence for
the group $G$ we get
$Ker\varphi_{B_{0}B_{1}B_{2}}=H^{(n)}_{B_{0}B_{1}B_{2}}(G).$
\end{proof}

Let us denote by $\aleph_{n}$ the following set\\
$$\{H^{(n)}_{B_{0}B_{1}B_{2}}(G)|\ B_{1}, B_{2} \ is\ a\ partition\
of\ the\ set\ N_{k}\setminus B_{0},\ 0\leq |B_{0}|\leq k-1,\
|G|=2n\}.$$
\begin{thm}\label{th2} Any normal subgroup of index $2n$ with
 $n=2^{i}(2s+1),\ i=\{0,1\},\ s\in\mathbb{N}$ has the form $H^{(n)}_{B_{0}B_{1}B_{2}}(G),\ |G|=2n$ i.e.,\\
$$\aleph_{n}=\{H|\ H\lhd G_{k},\ |G_{k}:H|=2n\}$$
\end{thm}

\begin{proof}  We will first prove that $\aleph_{n}\subseteq \{H|\ H\lhd G_{k},\
|G_{k}:H|=2n\}$. Let $G$ be a group with $|G|=2n$ and $ B_{1},
B_{2}$ a partition of the set $N_{k}\setminus B_{0},\ 0\leq
|B_{0}|\leq k-1.$ For $x=a_{i_{1}}a_{i_{2}}...a_{i_{n}}\in G_{k}$
it's sufficient to show that $x^{-1}H^{(n)}_{B_{0}B_{1}B_{2}}(G)$
$x\subseteq H^{(n)}_{B_{0}B_{1}B_{2}}(G).$ We have (as in the
proof of Proposition \ref{pr1}) if $y\in
H^{(n)}_{B_{0}B_{1}B_{2}}(G)$ then $\tilde{y}=e_{1},$ where
$e_{1}$ is the identity element of $G.$ Now we take any element
$z$ from the set $x^{-1}H^{(n)}_{B_{0}B_{1}B_{2}}(G)\ x.$ Then
$z=x^{-1}h\ x$ for some $h\in H^{(n)}_{B_{0}B_{1}B_{2}}(G).$ We have\\
$$\tilde{z}=\gamma_{n}(v_{n}(z))=\gamma_{n}\left(v_{n}(x^{-1}h\
 x)\right)=\gamma_{n}\left(v_{n}(x^{-1})v_{n}(h)v_{n}(x)\right)=$$
 $$=\gamma_{n}\left(v_{n}(x^{-1})\right)\gamma_{n}\left(v_{n}(h)\right)
 \gamma_{n}\left(v_{n}(x)\right)=\left(\gamma_{n}\left(v_{n}(x)\right)\right)^{-1}\gamma_{n}\left(v_{n}(h)\right)
 \gamma_{n}\left(v_{n}(x)\right).$$\\
From $\gamma_{n}\left(v_{n}(h)\right)=e_{1}$ we get
$\tilde{z}=e_{1}$ i.e., $z\in H^{(n)}_{B_{0}B_{1}B_{2}}(G).$ Hence\\
$$H^{(n)}_{B_{0}B_{1}B_{2}}(G)\in \{H|\ H\lhd G_{k},\ |G_{k}:H|=2n\}.$$\\
Now we'll prove that $\{H|\ H\lhd G_{k},\ |G_{k}:H|=2n\}\subseteq
\aleph_{n}.$ Let $H\lhd G_{k},\ |G_{k}:H|=2n.$ We consider the
\emph{natural homomorphism} $\phi:G_{k}\rightarrow G_{k}:H$ i.e.
the homomorphism given by $\phi(x)=xH,\ x\in G_{k}.$ There exist
$e,b_{1},b_{2},...b_{2n-1}$ such that $\phi:G_{k}\rightarrow \{H,
b_{1}H, ..., b_{2n-1}H\}$ is an epimorphism. Let $(\{H, b_{1}H,
..., b_{2n-1}H\},\ast)=\wp,$ i.e., $\wp$ is the factor group. If
we show that $\wp\in \Im_{2n}$ then the theorem will be proved.
Assuming that $\wp\notin \Im_{2n}$ then there are $c_{1},
c_{2},...,c_{q}\in \wp,\ q\geq 3$ such that
$\wp=<c_{1},...,c_{q}>.$ Clearly $<c_{1},c_{2}>$ is a subgroup of
$\wp$ and $|<c_{1},c_{2}>|$ is greater than three.\\

\textbf{\texttt{Case} $\textbf{n=2s+1}$.} By Lagrange's theorem
$|<c_{1},c_{2}>|\in \left\{m|\ \frac{2n}{m}\in
\mathbb{N}\right\}.$ If $e_{2}$ is the identity element of $\wp$
then from $c_{1}^{2}=e_{2}$ we take $|<c_{1},c_{2}>|=2n.$ Hence
$<c_{1},c_{2}>=\wp$ but $c_{3}\notin <c_{1},c_{2}>.$ Thus $\wp\in
\Im_{2n}.$\\

\textbf{\texttt{Case} $\textbf{n=2(2s+1)}$.} From Lagrange's
theorem $|<c_{1},c_{2}>|\in \left\{4, 2(2s+1), 4(2s+1)\right\}.$
Let $|<c_{1},c_{2}>|=4.$ If the number four isn't equal to one of
these numbers $|<c_{1},c_{3}>|$ or $|<c_{2},c_{3}>|$ then we'll
choose that pairs. If
$|<c_{1},c_{2}>|=|<c_{1},c_{3}>|=|<c_{2},c_{3}>|=4$ then $<c_{1},
c_{2}, c_{3}>$ is a group of order eight. Again by Lagrange's
theorem $|\wp|=2n$ must be divisible by 8 which is impossible.

Let $|<c_{1},c_{2}>|=2(2s+1).$ Then\\
$$<c_{1},c_{2}>=\{e,\ c_{1}, \ c_{2},\ c_{1}c_{2},\ c_{1}c_{2}c_{1},\ .\ .\ .\ ,\
\underbrace{c_{1}c_{2}...c_{1}}_{4s}\}=A.$$
 It's easy to check that
$$c_{3}A\cup A\subseteq \wp,\ \ \ \   c_{3}A\cap A=\emptyset,\
\ \ \ |c_{3}A\cup A|=|c_{3}A|+|A|=2n=|\wp|.$$

We then deduce that $c_{3}A\cup A=\wp$ but we'll show
$c_{3}c_{1}c_{3}\in \wp$ does not belong to $c_{3}A\cup A.$ Since
$c_{1}, c_{2}, c_{3}$ are generators we have
$c_{3}c_{1}c_{3}\notin<c_{1},c_{2}>$ so $c_{3}c_{1}c_{3}=xc_{3}$
with $x\in <c_{1}, c_{2}>$. But $x=c_{3}c_{1}\notin<c_{1},c_{2}>.$

If $|<c_{1}, c_{2}>|=4(2l+1)$ then $<c_{1}, c_{2}>=\wp$ but
$c_{3}\notin <c_{1}, c_{2}>.$ Hence $\wp\in \Im_{2n}.$ This
completes the proof.
\end{proof}

\subsection{Subgroups of index three.}

In this section "we will characterize the subgroups of index three
for the group representation of the Cayley tree.

 Let $(A_{1}, \  A_{2})$ be a partition of the set $N_{k}\backslash A_{0}, \,\
0\leq|A_{0}|\leq k-1.$ Put $m_{j}$ be a minimal element of
$A_{j},\,\ j\in\{1,2\}.$ Then we consider the function
$u_{A_{1}A_{2}}:\{a_{1},a_{2},...,a_{k+1}\}\rightarrow \{e,
a_{1}...,a_{k+1}\}$ given by
$$u_{A_{1}A_{2}}(x)=\left\{\begin{array}{ll}
e, \ \ \mbox{if} \ \ x=a_{i},\ i\in N_{k}\setminus (A_{1}\cup A_{2})\\[3mm]
a_{m_{j}}, \ \mbox{if}\ \ x=a_{i},\ i\in A_{j},\ j=\overline{1,2}.
\\ \end{array}\right.$$\\
Define $\gamma: <e, a_{m_{1}}, a_{m_{2}}> \ \rightarrow \{e,
a_{m_{1}}, a_{m_{2}}\}$ by the formula\\
$$\gamma(x)=\left\{\begin{array}{lll}
e  \ \mbox{if} \ \ x=e\\[2mm]
a_{m_{1}} \ \mbox{if} \ \ x\in\{a_{m_{1}},\ a_{m_{2}}a_{m_{1}}\}\\[2.5mm]
a_{m_{2}} \ \mbox{if} \ \ x\in\{a_{m_{2}},\
a_{m_{1}}a_{m_{2}}\}\\[2.5mm]
\gamma\left(a_{m_{i}}a_{m_{3-i}}a_{m_{i}}...\gamma(a_{m_{i}}a_{m_{3-i}})\right)
 \ \mbox{if} \ \ x=a_{m_{i}}a_{m_{3-i}}a_{m_{i}}...a_{m_{3-i}},\ l(x)\geq 3,\
 i=\overline{1,2}\\[2.5mm]
\gamma\left(a_{m_{i}}a_{m_{3-i}}a_{m_{i}}...\gamma(a_{m_{3-i}}a_{m_{i}})\right)
 \ \mbox{if} \ \ x=a_{m_{i}}a_{m_{3-i}}a_{m_{i}}...a_{m_{i}},\ l(x)\geq 3,\
 i=\overline{1,2}\\
 \end{array}\right.$$\\
For a partition $(A_{1}, \  A_{2})$ of the set $N_{k}\backslash
A_{0}, \,\ 0\leq|A_{0}|\leq k-1$ we consider the following set.\\
$$\Sigma_{A_{1}A_{2}}(G_{k})=\{x\in G_{k}|\
\gamma(u_{A_{1}A_{2}}(x))=e\}$$

\begin{lem}\label{lem1} Let $(A_{1}, \  A_{2})$
be a partition of the set $N_{k}\backslash A_{0}, \,\
0\leq|A_{0}|\leq k-1.$ Then $x\in \sum_{A_{1}A_{2}}(G_{k})$ if and
only if the number $l(u_{A_{1}A_{2}}(x))$ is divisible by 3.
\end{lem}
\begin{proof} Let $x=a_{i_{1}}a_{i_{2}}...a_{i_{n}}\in G_{k}$
and $l(u_{A_{1}A_{2}}(x))$ be odd (the even case is similar).
 Then we can write $u_{A_{1}A_{2}}(x)=a_{m_{i}}a_{m_{3-i}} ...
 a_{m_{i}}a_{m_{3-i}}a_{m_{i}}.$  $\gamma(u_{A_{1}A_{2}}(x))$ is
 equal to\\
 $$\gamma(a_{m_{i}}a_{m_{3-i}} ... a_{m_{i}}a_{m_{3-i}}a_{m_{i}})=\gamma\left(a_{m_{i}}a_{m_{3-i}} ...
 a_{m_{i}}\gamma(a_{m_{3-i}}a_{m_{i}})\right)=\gamma(\underbrace{a_{m_{i}}a_{m_{3-i}} ...
 a_{m_{3-i}}a_{m_{i}}}_{l(u_{A_{1}A_{2}}(x))-3})=...$$
 We continue this process until the length of the word
 $\gamma(u_{A_{1}A_{2}}(x))$ will be less than three. From
 $\gamma(u_{A_{1}A_{2}}(x))=e$ we deduce that
 $l(u_{A_{1}A_{2}}(x))$ is divisible by 3.

 Conversely if $l(u_{A_{1}A_{2}}(x))$ is divisible by 3 then $x$ is generated by the
 elements $a_{m_{1}}a_{m_{2}}a_{m_{1}}$ and
 $a_{m_{2}}a_{m_{1}}a_{m_{2}}.$ Since
 $\gamma(a_{m_{1}}a_{m_{2}}a_{m_{1}})=\gamma(a_{m_{2}}a_{m_{1}}a_{m_{2}})=e$
 we get $x\in \sum_{A_{1}A_{2}}(G_{k}).$
\end{proof}

\begin{pro}\label{pr2} For the group $G_{k}$ the following
equality holds
$$\{K|\ K\ is\ a\ subgroup\ of\ G_{k}\ of \ index \ 3\}=\ \ \ \ \ \ \ \ \ \ \ \ \ \ \ \
\ \ \ \ \ \ \ \ \ \ \ \ \ \ \ \ \ \ \ \ $$
$$\ \ \ \ \ \ \ \ \ \ \ \ \ \ \ \ \ \ \ \ =\{\Sigma_{A_{1}A_{2}}(G_{k})|\ A_{1},\ A_{2} \ is\ a\ partition\ of\ N_{k}\setminus A_{0}\}.$$
\end{pro}

\begin{proof} Let $K$ be a subgroup of the group
$G_{k}$ with $|G_{k}:K|=3$. Then there exist $p,q\in N_{k}$ such
that $|\{K,\ a_{p}K,\ a_{q}K\}|=3$. Put\\
$$A_{0}=\{i\in N_{k}|\ a_{i}\in K\},
 \ \ \ \ A_{1}=\{i\in N_{k}|\ a_{i}a_{p}\in K\},
 \ \ \ \ A_{2}=\{i\in N_{k}|\ a_{i}a_{q}\in K\}$$\\
From $|G_{k}:K|=3$ we conclude that $\{A_{1}, A_{2}\}$ is a
partition of $N_{k}\setminus A_{0}.$ Let $m_{i}$ be a minimal
element of $A_{i},\ i=\overline{1,2}$. If we show
$\sum_{A_{1}A_{2}}(G_{k})$ is a subgroup of $G_{k}$ (corresponding
to $K$) then it'll be given one to one correspondence between
given sets. For $x=a_{i_{1}}a_{i_{2}}...a_{i_{n}}\in
\sum_{A_{1}A_{2}}(G_{k}),\ y=a_{j_{1}}a_{j_{2}}...a_{j_{m}}\in
\sum_{A_{1}A_{2}}(G_{k})$ it is sufficient to show that
$xy^{-1}\in\sum_{A_{1}A_{2}}(G_{k}).$ Let\
$u_{A_{1}A_{2}}(x)=a_{m_{i}}a_{m_{3-i}}....a_{m_{s}},\
u_{A_{1}A_{2}}(y)=a_{m_{j}}a_{m_{3-j}}....a_{m_{t}}$, where
$(i,j,s,t)\in \{1,2\}^{4}$. Since $u_{A_{1}A_{2}}$ is a
homomorphism we have\\
$$l\left(u_{A_{1}A_{2}}(xy^{-1})\right)=\left\{\begin{array}{ll}
|l\left(u_{A_{1}A_{2}}(x))-l(u_{A_{1}A_{2}}(y^{-1})\right)|, \
 \ \mbox{if} \ \ s=t\\[2.5mm]
l\left(u_{A_{1}A_{2}}(x))+l(u_{A_{1}A_{2}}(y^{-1})\right),
\ \ \mbox{if} \ \ s\neq t.\\
\end{array}\right.$$\\

By Lemma \ref{lem1} we see that $l\left(u_{A_{1}A_{2}}(x)\right)$
is divisible by 3 and
$l\left(u_{A_{1}A_{2}}(y^{-1})\right)=l\left(\left(u_{A_{1}
A_{2}}(y)\right)^{-1}\right)$ is also divisible by 3, so
$l\left(u_{A_{1}A_{2}}(xy^{-1})\right)$ is a multiple of 3, which
shows that $xy^{-1}\in \sum_{A_{1}A_{2}}(G_{k})$. This completes
the proof.
\end{proof}

\end{document}